\documentclass[11pt,a4paper]{article}
\usepackage{amsmath,amsthm,color,eepic}
\newtheorem{theorem}{Theorem}

\newtheorem{lemma}{Lemma}
\newtheorem{false statement}{False statement}
\newtheorem{corollary}{Corollary}

\theoremstyle{definition}
\newtheorem{definition}{Definition}
\newtheorem{claim}{Claim}

\newtheorem{problem}{Problem}

\newcommand{\clc}{{\rm cl}^{\rm c}}
\newcommand{\clr}{{\rm cl}^{\rm r}}

\newcommand{\graphbox}[3]{\put(#1,#2){\put(0,0){\line(0,1){15}}
\put(0,0){\line(1,0){15}} \put(15,0){\line(0,1){15}}
\put(0,15){\line(1,0){15}} \put(2.5,2.5){#3}}}

\begin{document}
\title{\bf Heavy subgraphs, stability and hamiltonicity\thanks{Supported by
the NSFC (11271300) and the project NEXLIZ - CZ.1.07/2.3.00/30.0038.
}}
\date{}

\author{Binlong Li$^{1,3}$ and Bo Ning$^2$\thanks{Corresponding author. E-mail address: {\tt
bo.ning@tju.edu.cn (B. Ning)}.}\\[2mm]
\small $^1$ Department of Applied Mathematics, Northwestern Polytechnical University,\\
\small Xi'an, Shaanxi 710072, P.R. China\\
\small $^2$ Center for Applied Mathematics, Tianjin University,\\
\small Tianjin 300072, P.R. China\\
\small $^3$ European Centre of Excellence NTIS, 30614 Pilsen, Czech
Republic} \maketitle

\begin{center}
\begin{minipage}{130mm}
\small\noindent{\bf Abstract:} Let $G$ be a graph. Adopting the
terminology of Broersma et al. and \v{C}ada, respectively, we say
that $G$ is 2-heavy if every induced claw ($K_{1,3}$) of $G$
contains two end-vertices each one has degree at least $|V(G)|/2$;
and $G$ is o-heavy if every induced claw of $G$ contains two
end-vertices with degree sum at least $|V(G)|$ in $G$.
In this paper, we introduce a new concept, and say that
$G$ is \emph{$S$-c-heavy} if for a given graph
$S$ and every induced subgraph $G'$ of $G$ isomorphic to $S$ and
every maximal clique $C$ of $G'$, every non-trivial component of
$G'-C$ contains a vertex of degree at least $|V(G)|/2$ in $G$. In
terms of this concept, our original motivation that a theorem of Hu in
1999 can be stated as every 2-connected 2-heavy and $N$-c-heavy
graph is hamiltonian, where $N$ is the graph obtained from a
triangle by adding three disjoint pendant edges.
In this paper, we will characterize all connected
graphs $S$ such that every 2-connected o-heavy and $S$-c-heavy graph
is hamiltonian. Our work results in a different proof of a stronger
version of Hu's theorem. Furthermore, our main result improves or
extends several previous results.


\end{minipage}
\end{center}

\section{Introduction}

Throughout this paper, the graphs considered are undirected, finite
and simple (without loops and parallel edges). For terminology and
definition not defined here, we refer the reader to Bondy and Murty
\cite{BondyMurty}.

Let $G$ be a graph and $v$ be a vertex of $G$. The
\emph{neighborhood} of $v$ in $G$, denoted by $N_G(v)$, is the set
of neighbors of $v$ in $G$; and the \emph{degree} of $v$ in $G$,
denoted by $d_G(v)$, is the number of neighbors of $v$ in $G$. For
two vertices $u,v\in V(G)$, the \emph{distance} between $u$ and $v$
in $G$, denoted by $d_G(u,v)$, is the length of a shortest path
between $u$ and $v$ in $G$. When there is no danger of ambiguity, we
use $N(v)$, $d(v)$ and $d(u,v)$ instead of $N_G(v)$, $d_{G}(v)$ and
$d_{G}(u,v)$, respectively. For a subset $U$ of $V(G)$, we set
$N_U(v)=N(v)\cap U$, and $d_U(v)=|N_U(v)|$. For a subgraph $S$ of
$G$ such that $v\notin V(S)$, we use $N_S(v)$ and $d_S(v)$ instead
of $N_{V(S)}(v)$ and $d_{V(S)}(v)$, respectively.

Let $G$ be a graph and $G'$ be a subgraph of $G$. If $G'$ contains
all edges $xy\in E(G)$ with $x,y\in V(G')$, then $G'$ is an
\emph{induced subgraph} of $G$ (or a subgraph \emph{induced} by
$V(G')$). For a given graph $S$, the graph $G$ is \emph{$S$-free} if
$G$ contains no induced subgraph isomorphic to $S$. Note that if
$S_1$ is an induced subgraph of $S_2$, then an $S_1$-free graph is
also $S_2$-free.

The bipartite graph $K_{1,3}$ is the \emph{claw}. We use $P_i$
($i\geq 1$) and $C_i$ ($i\geq 3$) to denote the path and cycle of
order $i$, respectively. We denote by $Z_i$ ($i\geq 1$) the graph
obtained by identifying a vertex of a $C_3$ with an end-vertex of a
$P_{i+1}$; by $B_{i,j}$ ($i,j\geq 1$) the graph obtained by
identifying two vertices of a $C_3$ with the origins of a $P_{i+1}$
and a $P_{j+1}$, respectively; and by $N_{i,j,k}$ ($i,j,k\geq 1$)
the graph obtained by identifying the three vertices of a $C_3$ with
the origins of a $P_{i+1}$, a $P_{j+1}$ and a $P_{k+1}$,
respectively. In particular, we set $B=B_{1,1}$, $N=N_{1,1,1}$, and
$W=B_{1,2}$. (These three graphs are sometimes called the
\emph{bull}, the \emph{net} and the \emph{wounded}, respectively.)

To find sufficient conditions for hamiltonicity of graphs is a
standard topic. In particular, sufficient conditions for hamiltonicity of
graphs in terms of forbidden subgraphs have received much attention
from graph theorists. Following are some results in this area,
where the graphs $L_1$ and $L_2$ are shown in Figure 1.

\begin{theorem}\label{ThDuBrBeFa}
Let $G$ be a 2-connected graph.\\
(1) (\cite{DuffusJacbosonGould}) If $G$ is
claw-free and $N$-free, then $G$ is hamiltonian.\\
(2) (\cite{BroersmaVeldman}) If $G$ is
claw-free and $P_6$-free, then $G$ is hamiltonian.\\
(3) (\cite{Bedrossian}) If $G$ is
claw-free and $W$-free, then $G$ is hamiltonian.\\
(4) (\cite{FaudreeGouldRyjacekSchiermeyer}) If $G$ is claw-free and
$Z_3$-free, then $G$ is hamiltonian or $G=L_1$ or $L_2$.
\end{theorem}

\begin{center}
\begin{picture}(280,120)
\put(0,0){\multiput(20,30)(80,0){2}{\put(0,0){\circle*{4}}
\put(0,80){\circle*{4}} \put(20,40){\circle*{4}}
\put(0,0){\line(0,1){80}} \put(0,0){\line(1,2){20}}
\put(0,80){\line(1,-2){20}}} \put(20,30){\line(1,0){80}}
\put(20,30){\line(2,1){40}} \put(100,30){\line(-2,1){40}}
\put(20,110){\line(1,0){80}} \put(20,110){\line(2,-1){40}}
\put(100,110){\line(-2,-1){40}} \put(60,50){\circle*{4}}
\put(60,90){\circle*{4}} \put(80,70){\circle*{4}}
\put(60,50){\line(0,1){40}} \put(60,50){\line(1,1){20}}
\put(60,90){\line(1,-1){20}} \put(65,10){$L_1$}}

\put(140,0){\multiput(20,30)(80,0){2}{\put(0,0){\circle*{4}}
\put(0,80){\circle*{4}} \put(20,40){\circle*{4}}
\put(0,0){\line(1,2){20}} \put(0,80){\line(1,-2){20}}}
\put(20,30){\line(1,0){80}} \put(100,30){\line(0,1){80}}
\put(20,30){\line(2,1){40}} \put(100,30){\line(-2,1){40}}
\put(20,110){\line(1,0){80}} \put(20,110){\line(2,-1){40}}
\put(100,110){\line(-2,-1){40}} \put(60,50){\circle*{4}}
\put(60,90){\circle*{4}} \put(80,70){\circle*{4}}
\put(60,50){\line(0,1){40}} \put(60,50){\line(1,1){20}}
\put(60,90){\line(1,-1){20}} \put(65,10){$L_2$}}
\end{picture}

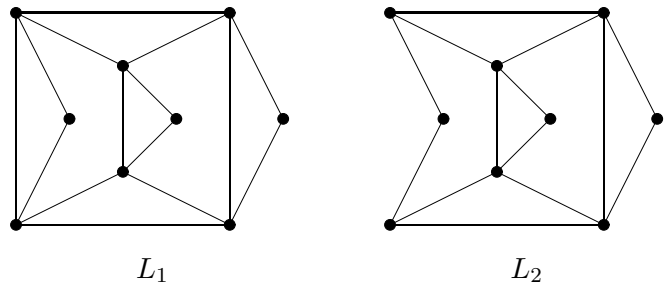
{\small Figure 1. Graphs $L_1$ and $L_2$.}
\end{center}

In 1991, Bedrossian \cite{Bedrossian} characterized all pairs of
forbidden subgraphs for a 2-connected graph to be hamiltonian, in
his Ph.D. Thesis. In 1997, Faudree and Gould \cite{FaudreeGould}
extended Bedrossian's result by proving the `only if' part based on
infinite families of non-hamiltonian graphs. Before showing the result
of Faudree and Gould, we first remark that the only connected graph
$S$ of order at least 3 such that the statement `every 2-connected
$S$-free graph is hamiltonian' holds, is $P_3$, see
\cite{FaudreeGould}. So in the following theorem, we only consider
the forbidden pairs excluding $P_3$.

\begin{theorem} [\cite{FaudreeGould}]\label{ThFaGo}
Let $R,S$ be connected graphs of order at least 3 with $R,S\neq P_3$
and let $G$ be a 2-connected graph of order $n\geq 10$. Then $G$
being $R$-free and $S$-free implies $G$ is hamiltonian if and only
if (up to symmetry) $R=K_{1,3}$ and
$S=P_4,P_5,P_6,C_3,Z_1,Z_2,Z_3,B,$ $N$ or $W$.
\end{theorem}

Degree condition is also an important type of sufficient conditions
for hamiltonicity of graphs. Let $G$ be a graph of order $n$. A
vertex $v\in V(G)$ is a \emph{heavy vertex} of $G$ if $d(v)\geq
n/2$; and a pair of vertices $\{u,v\}$ is a \emph{heavy pair} of $G$
if $uv\notin E(G)$ and $d(u)+d(v)\geq n$. In 1952, Dirac
\cite{Dirac} proved that every graph $G$ of order at least 3 is
hamiltonian if every vertex of $G$ is heavy. Ore \cite{Ore} improved
Dirac's result by showing that every graph $G$ of order at least 3
is hamiltonian if every pair of nonadjacent vertices is a heavy
pair. Fan \cite{Fan} further improved Ore's theorem by showing that every
2-connected graph $G$ is hamiltonian if every pair of vertices at
distance 2 of $G$ contains a heavy vertex.

It is natural to relax the forbidden subgraph conditions to ones
that the subgraphs are allowed, but some degree conditions are
restricted to the subgraphs. Early examples of this method used in
scientific papers can date back to 1990s
\cite{BedrossianChenSchelp,LiWeiGao,BroersmaRyjacekSchiermeyer}. In
particular, \v{C}ada \cite{Cada} introduced the class of
o-heavy graphs by restricting Ore's condition to every induced
claw of a graph. Li et al. \cite{LiRyjacekWangZhang} extended \v
Cada's concept of claw-o-heavy graphs to a general one.

Let $G'$ be an induced subgraph of $G$. Following
\cite{LiRyjacekWangZhang}, if $G'$ contains a heavy pair of $G$,
then $G'$ is an \emph{o-heavy subgraph} of $G$ (or $G'$ is
\emph{o-heavy} in $G$). For a given graph $S$, the graph $G$ is
\emph{$S$-o-heavy} if every induced subgraph of $G$ isomorphic to
$S$ is o-heavy. (It should be mentioned that \v{C}ada originally
named claw-o-heavy graphs as o-heavy graphs in \cite{Cada}.) Note
that an $S$-free graph is trivially $S$-o-heavy, and if $S_1$ is an
induced subgraph of $S_2$, then an $S_1$-o-heavy graph is also
$S_2$-o-heavy.

Li et al. \cite{LiRyjacekWangZhang} completely characterized pairs
of o-heavy subgraphs for a 2-connected graph to be hamiltonian,
which extends Theorem \ref{ThFaGo}. The main result in
\cite{LiRyjacekWangZhang} is given as follows.

\begin{theorem}[\cite{LiRyjacekWangZhang}]\label{ThLiRyWaZh}
Let $R$ and $S$ be connected graphs of order at least 3 with
$R,S\neq P_3$ and let $G$ be a 2-connected graph. Then $G$ being
$R$-o-heavy and $S$-o-heavy implies $G$ is hamiltonian if and only
if (up to symmetry) $R=K_{1,3}$ and $S=P_4,P_5,C_3,Z_1,Z_2,B,N$ or
$W$.
\end{theorem}

Following \cite{NingZhang}, we introduce another type of heavy
subgraph condition motivated by Fan's condition \cite{Fan}. Let $G$ be a graph
and $G'$ be an induced subgraph of $G$. If for each two vertices
$u,v\in V(G')$ with $d_{G'}(u,v)=2$, either $u$ or $v$ is heavy in
$G$, then $G'$ is an \emph{f-heavy subgraph} of $G$ (or $G'$ is
\emph{f-heavy} in $G$). For a given graph $S$, the graph $G$ is
\emph{$S$-f-heavy} if every induced subgraph of $G$ isomorphic to
$S$ is f-heavy. A claw-f-heavy graph is also called a \emph{2-heavy}
graph (see \cite{BroersmaRyjacekSchiermeyer}).

Note that an $S$-free graph is trivially $S$-f-heavy, but in
general, an $S_1$-f-heavy graph is not necessarily $S_2$-f-heavy
when $S_1$ is an induced subgraph of $S_2$. In Figure 2, we show the
implication relations among the conditions being $S$-f-heavy for the
graphs $S$ listed in Theorem \ref{ThFaGo}.

\begin{center}
\begin{picture}(175,95)

\graphbox{20}{40}{$P_3$} \graphbox{50}{10}{$W$}
\graphbox{80}{10}{$Z_2$} \graphbox{110}{10}{$Z_3$}
\graphbox{50}{40}{$P_4$} \graphbox{80}{40}{$P_5$}
\graphbox{110}{40}{$P_6$} \graphbox{140}{40}{$C_3$}
\graphbox{50}{70}{$Z_1$} \graphbox{80}{70}{$B$}
\graphbox{110}{70}{$N$}

\put(35,40){\vector(1,-1){15}} \put(65,40){\vector(1,-1){15}}
\put(95,40){\vector(1,-1){15}} \put(65,25){\line(5,1){75}}
\put(140,40){\vector(1,0){0}} \put(95,25){\vector(3,1){45}}
\put(125,25){\vector(1,1){15}} \put(35,47.5){\vector(1,0){15}}
\put(65,47.5){\vector(1,0){15}} \put(95,47.5){\vector(1,0){15}}
\put(125,47.5){\vector(1,0){15}} \put(35,55){\vector(1,1){15}}
\put(65,77.5){\vector(1,0){15}} \put(95,77.5){\vector(1,0){15}}
\put(65,55){\vector(3,1){45}} \put(125,70){\vector(1,-1){15}}

\end{picture}

{\small Figure 2. $S_1\rightarrow S_2$: Being $S_1$-f-heavy implies
being $S_2$-f-heavy}
\end{center}

We remark that f-heavy conditions cannot compare with o-heavy
conditions in general. For example, every $P_3$-o-heavy graph is
$P_3$-f-heavy; and every claw-f-heavy graph is claw-o-heavy, but for
the conditions being $N$-o-heavy and being $N$-f-heavy, no one can
imply the other.

Motivated by Theorem \ref{ThLiRyWaZh}, Ning and Zhang
\cite{NingZhang} characterized pairs of f-heavy subgraphs for a
2-connected graph to be hamiltonian, which not only is a new
extension of Theorem \ref{ThFaGo} but also unifies some previous
theorems in \cite{BedrossianChenSchelp,ChenWeiZhang,LiWeiGao}.

\begin{theorem}[\cite{NingZhang}]\label{ThNiZh}
Let $R$ and $S$ be connected graphs with $R,S\neq P_3$ and let $G$
be a 2-connected graph of order $n\geq 10$. Then $G$ being
$R$-f-heavy and $S$-f-heavy implies $G$ is hamiltonian if and only
if (up to symmetry) $R=K_{1,3}$ and $S=P_4,P_5,P_6,Z_1,Z_2,Z_3,B,N$
or $W$.
\end{theorem}

Now we will put our views to another new sufficient condition for
hamiltonicity of graphs due to Hu \cite{Hu}. Some previous theorems
can be obtained from Hu's theorem as corollaries (see
\cite{BedrossianChenSchelp,LiWeiGao}).

\begin{theorem}[\cite{Hu}]\label{ThHu}
Let $G$ be a 2-connected graph. If $G$ is 2-heavy
and every induced $P_4$ in an induced $N$ of $G$
contains a heavy vertex, then $G$ is hamiltonian.
\end{theorem}

In fact, we can see that the cases $S=Z_1,B,N$ in Theorem \ref{ThNiZh} can be
deduced from Hu's theorem. This motivates us to consider the
counterpart results for other subgraphs. Armed with this idea, we
first propose the following definition.

\begin{definition}
Let $G$ be a graph and $G'$ be an induced subgraph of $G$. If for
every maximal clique $C$ of $G'$, each nontrivial component of
$G'-C$ contains a heavy vertex of $G$, then $G'$ is a
\emph{clique-heavy} (or in short, \emph{c-heavy}) subgraph of $G$.
For a given graph $S$, $G$ is \emph{$S$-c-heavy} if every induced
subgraph of $G$ isomorphic to $S$ is c-heavy.
\end{definition}

In Figure 3, we show the implication relations of the conditions
being $S$-c-heavy for the graphs $S$ listed in Theorem \ref{ThFaGo}.

\begin{center}
\begin{picture}(175,95)

\graphbox{20}{70}{$Z_1$} \graphbox{50}{70}{$B$}
\graphbox{80}{70}{$N$} \graphbox{20}{40}{$P_4$}
\graphbox{50}{40}{$P_5$} \graphbox{80}{40}{$P_6$}
\graphbox{110}{40}{$C_3$} \graphbox{140}{40}{$P_3$}
\graphbox{20}{10}{$Z_2$} \graphbox{50}{10}{$W$}
\graphbox{80}{10}{$Z_3$}

\put(35,77.5){\vector(1,0){15}} \put(65,77.5){\vector(1,0){15}}
\put(35,47.5){\vector(1,0){15}} \put(65,47.5){\vector(1,0){15}}
\put(35,17.5){\vector(1,0){15}} \put(65,25){\vector(3,1){45}}
\put(35,55){\vector(1,1){15}} \put(57.5,40){\vector(0,-1){15}}
\put(95,70){\vector(1,-1){15}} \put(95,47.5){\vector(1,0){15}}
\put(95,25){\vector(1,1){15}} \put(125,47.5){\vector(1,0){15}}
\put(125,47.5){\vector(-1,0){0}}

\end{picture}

{\small Figure 3. $S_1\rightarrow S_2$: Being $S_1$-c-heavy implies
being $S_2$-c-heavy.}
\end{center}

So Theorem \ref{ThHu} can be stated as every 2-connected
claw-f-heavy and $N$-c-heavy graph is hamiltonian. As we will show
below, this can be extended to that every 2-connected claw-o-heavy
and $N$-c-heavy graph is hamiltonian.

We remark that saying a graph is claw-c-heavy is meaningless (if we
remove a maximal clique from a claw, then only isolated vertices
remain). Motivated by Theorems \ref{ThFaGo}, \ref{ThLiRyWaZh} and
\ref{ThNiZh}, we naturally propose the following problem.

\begin{problem}\label{PrcHeavy}
Which connected graphs $S$ imply that every 2-connected
claw-free (or claw-f-heavy or claw-o-heavy) and $S$-c-heavy graph is
hamiltonian?
\end{problem}

The solution to Problem 1 is one of the main results in this paper.

\begin{theorem}\label{ThcHeavy}
Let $S$ be a connected graph of order at least 3 and let $G$ be a
2-connected claw-o-heavy graph of order $n\geq 10$. Then $G$ being
$S$-c-heavy implies $G$ is hamiltonian if and only if $S=P_4,P_5,P_6,Z_1,Z_2,Z_3,B,N$ or $W$.
\end{theorem}

Note that the only subgraphs appearing in Theorem \ref{ThFaGo} but
missed here are $P_3$ and $C_3$. Also note that every graph is
$P_3$-c-heavy and $C_3$-c-heavy and there exist 2-connected
claw-free graphs which are non-hamiltonian. By Theorem \ref{ThFaGo}
and the fact that every claw-free (claw-f-heavy) graph is
claw-o-heavy, we can see that Theorem \ref{ThcHeavy} gives a complete
solution to Problem \ref{PrcHeavy}.

We point out that a special case of our work results in a new proof of a stronger version of
Theorem \ref{ThHu}.

\begin{theorem}
Let $G$ be a 2-connected graph. If $G$ is claw-o-heavy and $N$-c-heavy, then $G$ is hamiltonian.
\end{theorem}

Some previous theorems can also be obtained from this theorem
as corollaries in a unified way.

\begin{corollary}[\cite{Hu}]
Let $G$ be a graph. If $G$ is claw-f-heavy and $N$-c-heavy, then $G$ is hamiltonian.
\end{corollary}

\begin{corollary}[\cite{NingZhang}]
Let $G$ be a graph. If $G$ is claw-o-heavy and $N$-f-heavy, then $G$ is hamiltonian.
\end{corollary}

\begin{corollary}[\cite{LiWeiGao}]
Let $G$ be a graph. If $G$ is claw-f-heavy and $B$-f-heavy, then $G$ is hamiltonian.
\end{corollary}

\begin{corollary}[\cite{BedrossianChenSchelp}]
Let $G$ be a graph. If $G$ is claw-f-heavy and $Z_1$-f-heavy, then $G$ is hamiltonian.
\end{corollary}

We remark that our methods used here are completely different from
the ones in \cite{Hu,LiRyjacekWangZhang,NingZhang}. We mainly use
the claw-o-heavy closure theory introduced by \v{C}ada
\cite{Cada}, and many other results from the area of forbidden
subgraphs. However, our technique here is new, and it is heavily
dependent on some new concepts and tools developed by us
recently. (See Lemma 7 in Sec.2 for example.) We point out that this
is the first time to deal with Hamiltonicity of graphs under pairs
of heavy subgraph conditions by using c-Closure theory
systemically, compared with several previous works in
\cite{BedrossianChenSchelp,LiWeiGao,Hu,ChenWeiZhang,LiRyjacekWangZhang,NingZhang,NingLiZhang}.

The rest of this paper is organized as follows. In Section 2, we
will present necessary and additional preliminaries (including the introduction to
claw-free closure theory, claw-o-heavy closure theory and a useful
theorem of Brousek). In Section 3, in the spirit of some previous works
of Brousek et al. \cite{BrousekRyjacekFavaron}, we will study the
stability of some subclasses of the class of claw-o-heavy graphs. In
Section 4, by using the closure theory and a previous result of
Brousek \cite{Brousek}, we give the proof of Theorem \ref{ThcHeavy}.
In Section 5, one useful remark is given to conclude this paper.

\section{Preliminaries}

The main tools in our paper are
two kinds of closure theories introduced by Ryj\'a\v{c}ek
\cite{Ryjacek} and \v{C}ada \cite{Cada}, respectively. These two
closure theories are used to study hamiltonian properties of
claw-free graphs and claw-o-heavy graphs, respectively. We will give
some terminology and notation with a prefix or superscript r or c,
respectively, to distinguish them.

\subsection*{r-Closure theory.}

Let $G$ be a claw-free graph and $x$ be a vertex of $G$. Following
\cite{Ryjacek}, we call $x$ an \emph{r-eligible vertex} of $G$ if
$N(x)$ induces a connected graph in $G$ but not a complete graph.
The \emph{completion of $G$ at $x$}, denoted by $G'_x$, is the graph
obtained from $G$ by adding all missing edges $uv$ with $u,v\in
N(x)$.

\begin{lemma}[\cite{Ryjacek}]\label{LeRy}
Let $G$ be a claw-free graph and $x$ be an r-eligible vertex of $G$.
Then\\
(1) the graph $G'_x$ is claw-free; and\\
(2) the circumferences of $G'_x$ and $G$ are equal.
\end{lemma}

The \emph{r-closure} of a claw-free graph $G$, denoted by $\clr(G)$,
is defined by a sequence of graphs $G_1,G_2,\ldots,G_t$, and
vertices $x_1,x_2\ldots,x_{t-1}$ such that \\
(1) $G_1=G$, $G_t=\clr(G)$; \\
(2) $x_i$ is an r-eligible vertex of $G_i$, $G_{i+1}=(G_i)'_{x_i}$,
$1\leq i\leq t-1$; and\\
(3) $\clr(G)$ has no r-eligible vertices.

A claw-free graph $G$ is \emph{r-closed} if $G$ has no r-eligible
vertices, i.e., if $\clr(G)=G$.

\begin{theorem}[\cite{Ryjacek}]\label{ThRy}
Let $G$ be a claw-free graph. Then\\
(1) the r-closure $\clr(G)$ is well defined;\\
(2) there is a $C_3$-free graph $H$ such that $\clr(G)$ is the line
graph of $H$; and\\
(3) the circumferences of $\clr(G)$ and $G$ are equal.
\end{theorem}

It is not difficult to get the following (see
\cite{BrousekRyjacekFavaron}).

\begin{lemma}[\cite{BrousekRyjacekFavaron}]\label{LeBrRyFa}
Let $G$ be a claw-free graph. Then $\clr(G)$ is a
$K_{1,1,2}$-free supergraph of $G$ with the least number of edges.
\end{lemma}

Following \cite{BrousekRyjacekFavaron}, we say a family
$\mathcal{G}$ of graphs is \emph{stable under the r-closure} (or shortly,
r-stable) if for every graph in $\mathcal{G}$, its r-closure is also
in $\mathcal{G}$. From Theorem \ref{ThRy}, we can see that the class
of all claw-free hamiltonian graphs and the class of all claw-free
non-hamiltonian graphs are r-stable.

\subsection*{c-Closure theory.}
Let $G$ be a claw-o-heavy graph and let $x\in V(G)$. Let $G'$ be the
graph obtained from $G$ by adding the missing edges $uv$ with
$u,v\in N(x)$ and $\{u,v\}$ is a heavy pair of $G$. We call $x$ a
\emph{c-eligible} vertex of $G$ if $N(x)$ is not a clique of $G$
and one of the following is true:\\
(1) $G'[N(x)]$ is connected; or\\
(2) $G'[N(x)]$ consists of two disjoint cliques $C_1$ and $C_2$, and
$x$ is contained in a heavy pair $\{x,z\}$ of $G$ such that
$zy_1,zy_2\in E(G)$ for some $y_1\in C_1$ and $y_2\in C_2$.\\
Note that if $G$ is claw-free, then an r-eligible vertex is also
c-eligible.

\begin{lemma}[\cite{Cada}]\label{LeCa}
Let $G$ be a claw-o-heavy graph and $x$ be a c-eligible vertex of $G$. Then\\
(1) for every vertex $y\in N(x)$, $d_{G'_x}(y)\geq d_{G'_x}(x)$;\\
(2) the graph $G'_x$ is claw-o-heavy; and\\
(3) the circumferences of $G'_x$ and $G$ are equal.
\end{lemma}

The \emph{c-closure} of a claw-o-heavy graph $G$, denoted by
$\clc(G)$, is defined by a sequence of graphs $G_1,G_2,\ldots,G_t$,
and
vertices $x_1,x_2\ldots,x_{t-1}$ such that \\
(1) $G_1=G$, $G_t=\clc(G)$; \\
(2) $x_i$ is a c-eligible vertex of $G_i$, $G_{i+1}=(G_i)'_{x_i}$,
$1\leq i\leq t-1$; and\\
(3) $\clc(G)$ has no c-eligible vertices.

\begin{theorem}[\cite{Cada}]\label{ThCa}
Let $G$ be a claw-o-heavy graph. Then\\
(1) the c-closure $\clc(G)$ is well defined;\\
(2) there is a $C_3$-free graph $H$ such that $\clc(G)$ is the line
graph of $H$; and\\
(3) the circumferences of $\clc(G)$ and $G$ are equal.
\end{theorem}

A claw-o-heavy graph $G$ is \emph{c-closed} if $\clc(G)=G$. Note
that every line graph is claw-free (see \cite{Beineke}). This
implies that $\clc(G)$ is a claw-free graph. Also note that for a
claw-free graph, an r-eligible vertex is also c-eligible. This
implies that every c-closed graph is also r-closed.

Similarly as the case of r-closure, we say a family $\mathcal{G}$ of
graphs is stable under the c-closure (or shortly, c-stable) if for
every graph in $\mathcal{G}$, its c-closure is also in
$\mathcal{G}$.

The following lemma is an obvious but important fact, which can
be deduced from Lemma 14 in \cite{Cada} easily.
\begin{lemma}[\cite{Cada}]\label{LeHeavyPair}
Let $G$ be a claw-o-heavy graph. Then $\clc(G)$ has no heavy pair.
\end{lemma}

Here we list some new concepts introduced by us recently \cite{NingLiZhang}.
Let $G$ be a claw-o-heavy graph and $C$ be a maximal clique of
$\clc(G)$. We call $G[C]$ a \emph{region} of $G$. For a vertex $v$
of $G$, we call $v$ an \emph{interior vertex} if it is contained in
only one region, and a \emph{frontier vertex} if it is contained in
two distinct regions.

A graph $G$ is nonseparable if it is connected and has no
cut-vertex (i.e., either $G$ is 2-connected, or $G=K_1$ or $K_2$).
The following useful lemma originally appeared as Lemma 2 in
\cite{NingLiZhang}, and it plays the crucial role of our proofs.

\begin{lemma}[\cite{NingLiZhang}]\label{LeRegion}
Let $G$ be a claw-o-heavy graph and $R$ be a region of $G$. Then\\
(1) $R$ is nonseparable;\\
(2) if $v$ is a frontier vertex of $R$, then $v$ has an interior
neighbor in $R$ or $R$ is complete and has no interior vertices; and\\
(3) for any two vertices $u,v\in R$, there is an induced path of $G$
from $u$ to $v$ such that every internal vertex of the path is an
interior vertex of $R$.
\end{lemma}

Following \cite{Brousek}, we  define $\mathcal{P}$ to be the class
of graphs obtained from two vertex-disjoint triangles
$a_1a_2a_3a_1$ and $b_1b_2b_3b_1$ by joining every pair of
vertices $\{a_i,b_i\}$ by a path $P_{k_i}$, where $k_i\geq 3$ or by a triangle. We use
$P_{x_1,x_2,x_3}$ to denote the graph in $\mathcal{P}$, where
$x_i=k_i$ if $a_i$ and $b_i$ are joined by a path $P_{k_i}$, and
$x_i=T$ if $a_i$ and $b_i$ are joined by a triangle. Note that
$L_1=P_{T,T,T}$ and $L_2=P_{3,T,T}$.

We give the following useful result to finish this section.

\begin{theorem}[\cite{Brousek}]\label{ThBr}
Every non-hamiltonian 2-connected claw-free graph contains an
induced subgraph $G'\in \mathcal{P}$.
\end{theorem}

\section{Stable classes under closure operation}

Brousek et al. \cite{BrousekRyjacekFavaron} studied the graphs $S$
such that the class of claw-free $S$-free graphs is r-stable. Before
we present their result, we first remark that if $S$ contains an
induced claw or an induced $K_{1,1,2}$, then the class of claw-free
and $S$-free graphs is trivially r-stable by Lemma \ref{LeBrRyFa}.
So in the following theorem we assume that $S$ is claw-free and
$K_{1,1,2}$-free.

\begin{theorem}[\cite{BrousekRyjacekFavaron}]\label{ThBrRyFa}
Let $S$ be a connected claw-free and $K_{1,1,2}$-free graph of order
at least 3. Then the class of claw-free and $S$-free graphs is
r-stable, if and only if
$$S\in\{C_3,H\}\cup\{P_i: i\geq 3\}\cup\{Z_i: i\geq
1\}\cup\{N_{i,j,k}: i,j,k\geq 1\}.$$
\end{theorem}

\begin{center}
\begin{picture}(120,70)
\thicklines

\put(0,0){\put(20,10){\circle*{4}} \put(20,60){\circle*{4}}
\put(60,35){\circle*{4}} \put(100,10){\circle*{4}}
\put(100,60){\circle*{4}} \put(20,10){\line(0,1){50}}
\put(20,10){\line(8,5){80}} \put(20,60){\line(8,-5){80}}
\put(100,10){\line(0,1){50}}}

\end{picture}

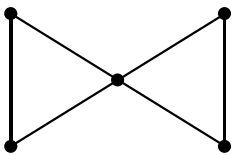
{\small Figure 4. Graph $H$ (hourglass).}
\end{center}

In the spirit of previous works of Brousek et al.
\cite{BrousekRyjacekFavaron}, we will consider the c-stability of
the class of claw-o-heavy and $S$-c-heavy graphs. Before showing our
results about this topic, we first remark the following trivial
facts:

If $S$ is the join of a complete graph and an empty graph
(specially, if $S$ is a complete graph or a star), then for every
maximal clique $C$ of $S$, $S-C$ has only trivial components. Thus
by our definition, every graph will be $S$-c-heavy. Moreover, by our
definition of c-stability, the class of claw-o-heavy and $S$-c-heavy
graphs is c-stable. In the following, we will characterize all the
other graphs $S$ such that the class of claw-o-heavy and $S$-c-heavy
graphs is c-stable.

\begin{center}
\begin{picture}(250,125)
\thicklines

\put(0,80){\multiput(20,30)(30,0){5}{\put(0,0){\circle*{4}}}
\put(20,30){\line(1,0){60}} \put(110,30){\line(1,0){30}}
\qbezier[4](80,30)(95,30)(110,30) \put(17,33){$a_1$}
\put(47,33){$a_2$} \put(77,33){$a_3$} \put(107,33){$a_{i-1}$}
\put(137,33){$a_i$} \put(75,10){$P_i$}}

\put(5,0){\put(20,30){\circle*{4}} \put(20,80){\circle*{4}}
\put(20,30){\line(0,1){50}} \put(20,30){\line(1,1){25}}
\put(20,80){\line(1,-1){25}}
\multiput(45,55)(60,0){2}{\multiput(0,0)(30,0){2}{\circle*{4}}
\put(0,0){\line(1,0){30}}} \qbezier[4](75,55)(90,55)(105,55)
\put(12,78){$b$} \put(12,28){$c$} \put(42,58){$a$}
\put(72,58){$a_1$} \put(102,58){$a_{i-1}$} \put(132,58){$a_i$}
\put(70,10){$Z_i$}}

\put(160,0){\multiput(20,30)(50,0){2}{\multiput(0,0)(0,30){2}{\put(0,0){\circle*{4}}}
\put(0,0){\line(0,1){30}}}
\multiput(45,85)(0,30){2}{\put(0,0){\circle*{4}}}
\put(45,85){\line(0,1){30}} \put(20,60){\line(1,0){50}}
\put(20,60){\line(1,1){25}} \put(70,60){\line(-1,1){25}}
\put(48,83){$a$} \put(48,113){$a_1$} \put(10,58){$b$}
\put(10,28){$b_1$} \put(73,58){$c$} \put(73,28){$c_1$}
\put(40,10){$N$}}

\end{picture}

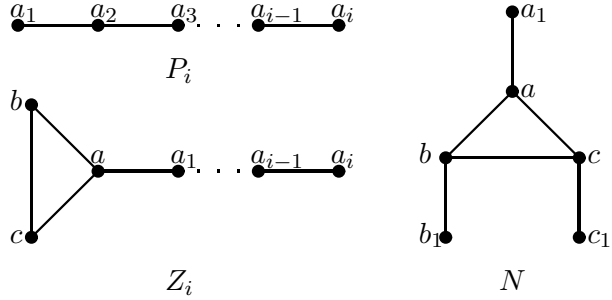
{\small Figure 5. Graphs $P_i$, $Z_i$ and $N$.}
\end{center}

For a vertex $x$ of a graph $G$, we set $B_G(x)=\{uv: u,v\in N(x)$
and $uv\notin E(G)\}$. For convenience, we say a vertex or a pair of
nonadjacent vertices is \emph{light} if it is not heavy.

\begin{theorem}\label{ThSPi}
Let $G$ be a claw-o-heavy and $P_i$-c-heavy graph, $i\geq 4$, and
$x$ be a c-eligible vertex of $G$. Then $G'_x$ is $P_i$-c-heavy.
\end{theorem}

\begin{proof}
Let $P$ be an induced $P_i$ of $G'_x$. We denote the vertices of $P$
as in Figure 5, and will prove that one vertex of $\{a_1,a_2\}$ is
heavy in $G'_x$ and one vertex of $\{a_{i-1},a_i\}$ is heavy in
$G'_x$. Note that $d_{G'_x}(v)\geq d(v)$ for every vertex $v\in
V(G)$. If $P$ is also an induced subgraph of $G$, then $P$ is
c-heavy in $G$, and then, is c-heavy in $G'_x$. So we assume that
$P$ is not an induced subgraph of $G$, which implies that $E(P)\cap
B_G(x)\neq\emptyset$. Suppose that $a_ja_{j+1}$ is an edge in
$E(P)\cap B_G(x)$, where $1\leq j\leq i-1$.

Since $N(x)$ is a clique in $G'_x$, $N(x)\cap V(P)=\{a_j,a_{j+1}\}$
and there is only one edge in $E(P)\cap B_G(x)$. If $j\geq 2$, then
$P'=a_1a_2\cdots a_jxa_{j+1}\cdots a_{i-1}$ is an induced $P_i$ of
$G$. Since $G$ is $P_i$-c-heavy, one vertex of $\{a_1,a_2\}$ is
heavy in $G$, and then, is heavy in $G'_x$. If $j=1$, then
$P'=a_1xa_2\cdots a_{i-1}$ is an induced $P_i$ of $G$. Thus one
vertex of $\{a_1,x\}$ is heavy in $G$. Note that $d_{G'_x}(a_1)\geq
d_{G'_x}(x)=d(x)$ (see Lemma \ref{LeCa}). Thus $a_1$ is heavy in
$G'_x$. Hence in any case, we have shown that one vertex of
$\{a_1,a_2\}$ is heavy in $G'_x$. By the symmetry, we can prove that
one vertex of $\{a_{i-1},a_i\}$ is heavy in $G'_x$.
\end{proof}

Note that every c-closed graph has no heavy pairs, and note that
every c-heavy $P_i$ with $i\geq 5$ must have a heavy pair. By
Theorem \ref{ThSPi}, we have

\begin{corollary}\label{CoFPi}
Let $G$ be a claw-o-heavy and $P_i$-c-heavy graph with $i\geq 5$.
Then $\clc(G)$ is $P_i$-free.
\end{corollary}

\begin{corollary}\label{CoSPi}
For $i\geq 3$, the class of claw-o-heavy and $P_i$-c-heavy graphs is
c-stable.
\end{corollary}

There are no counterpart results of Theorem \ref{ThSPi} for the
graph $Z_i$. In fact, there exist claw-free and $Z_i$-free graphs
$G$ with an r-eligible vertex $x$ such that $G'_x$ is not
$Z_i$-free, see \cite{BrousekRyjacekFavaron}. However, we can prove
that the class of claw-o-heavy and $Z_i$-c-heavy graphs is also
c-stable for $i\neq 2$.

\begin{theorem}\label{ThSZ1}
Let $G$ be a claw-o-heavy and $Z_1$-c-heavy graph. Then $\clc(G)$ is
also $Z_1$-c-heavy.
\end{theorem}

\begin{proof}
Let $Z$ be an induced $Z_1$ in $\clc(G)$. We denote the vertices of
$Z$ as in Figure 5. We will prove that either $b$ or $c$ is heavy.

\begin{claim}\label{ClNxH}
Let $R$ be a region of $G$ and $x\in V(R)$ be a frontier vertex. If
$y,y'$ are two neighbors of $x$ in $R$, then one vertex in
$\{y,y'\}$ is heavy in $G$.
\end{claim}

\begin{proof}
Let $z$ be a neighbor of $x$ in $G-R$. Clearly $yz,y'z\notin E(G)$.
If $yy'\in E(G)$, then the subgraph of $G$ induced by $\{x,y,y',z\}$
is a $Z_1$. Since $G$ is $Z_1$-c-heavy, either $y$ or $y'$ is heavy
in $G$. Now we assume that $yy'\notin E(G)$. Then the subgraph of
$G$ induced by $\{x,y,y',z\}$ is a claw. Note that $\{y,z\}$ and
$\{y',z\}$ are not heavy pairs in $\clc(G)$, and then, are not heavy
pairs in $G$. This implies that $\{y,y'\}$ is a heavy pair of $G$.
Thus either $y$ or $y'$ is heavy in $G$.
\end{proof}

Suppose that both $b$ and $c$ are light. Let $R$ be the region of
$G$ containing $\{a,b,c\}$. Note that $R$ is a clique in $\clc(G)$.
If $|V(R)|\geq |V(G)|/2+1$, then $b$ is heavy in $\clc(G)$, a
contradiction. So we assume that $|V(R)|\leq(|V(G)|+1)/2$. This
implies that every interior vertex of $R$ is light in $\clc(G)$, and
also, in $G$.

If $R$ has no interior vertex, then by Lemma \ref{LeRegion}, $R$ is
a clique in $G$. By Claim \ref{ClNxH}, either $b$ or $c$ is heavy in
$G$, a contradiction. So we assume that $R$ has an interior vertex.
By Lemma \ref{LeRegion}, $R$ has an interior vertex adjacent to $a$.
Since $a$ has at least two neighbors in $R$, we may choose two
neighbors $x,y$ of $a$ in $R$ such that $x$ is an interior vertex of
$R$. Note that $x$ is light in $G$. By Claim \ref{ClNxH}, $y$ is
heavy in $G$. Recall that $b,c$ and every interior vertex of $R$ are
light. Hence $y\neq b,c$ and $y$ is a frontier vertex of $R$.

If both $by$ and $cy$ are in $E(G)$, then by Claim \ref{ClNxH},
either $b$ or $c$ is heavy in $G$, a contradiction. So we conclude
that $by\notin E(G)$ or $cy\notin E(G)$.

If $d_{G-R}(y)=1$, then $d(y)=d_R(y)+1\leq|V(R)|-2+1\leq(n-1)/2$.
Hence $y$ is light in $G$, a contradiction. So we conclude that
$d_{G-R}(y)\geq 2$. Also note that $d_R(y)\geq 2$ by Lemma
\ref{LeRegion}. Let $x',x''$ be two vertices in $N_R(y)$ and
$y',y''$ be two vertices in $N_{G-R}(y)$. By Claim \ref{ClNxH}, one
vertex of $\{x',x''\}$ is heavy in $G$, and one vertex of
$\{y',y''\}$ is heavy in $G$. We assume without loss of generality
that $x',y'$ are heavy in $G$. Then $\{x',y'\}$ is a heavy pair in
$G$, and also is a heavy pair of $\clc(G)$, a contradiction.
\end{proof}

\begin{theorem}\label{ThSZi}
Let $G$ be a claw-o-heavy and $Z_i$-c-heavy graph with $i\geq 3$.
Then $\clc(G)$ is $Z_i$-free.
\end{theorem}

\begin{proof}

The proof is almost the same as the proof of Lemma 3 in
\cite{NingLiZhang}. The only difference occurs when we
find an induced $Z_i$ in $\clc(G)$, instead of a $Z_3 $ as
done in the proof of Lemma 3 in \cite{NingLiZhang}, and when we
use the c-heavy condition, instead of the f-heavy condition.
But we still shall carry it in full,
due to some specific details and the integrity of this paper.
Now we give the proof along the outline in
\cite{NingLiZhang} step by step.

Suppose the contrary. Let $Z$ be an induced $Z_i$ in $\clc(G)$. We
denote the vertices of $Z$ as in Figure 5. Let $R$ be the region of
$G$ containing $\{a,b,c\}$. Proofs of the first two claims are almost the
same as Claims 1, 2 in the proof of Lemma 3 in \cite{NingLiZhang}.

\setcounter{claim}{0}
\begin{claim}\cite[Claim 1 in the proof of Lemma 3]{NingLiZhang}\\
$|N_R(a_2)\cup N_R(a_3)|\leq 1$.
\end{claim}

\begin{proof}
Note that every vertex in $G-R$ has at most one neighbor in $R$. If
$N_R(a_2)=\emptyset$, then the assertion is obviously true. Now we
assume that $N_R(a_2)\neq\emptyset$. Let $x$ be the vertex in
$N_R(a_2)$. Clearly $x\neq a$ and $a_1x\notin E(\clc(G))$. If
$a_3x\notin E(\clc(G))$, then $\{a_2,a_1,a_3,x\}$ induces a claw in
$\clc(G)$, a contradiction. This implies that $a_3x\in E(\clc(G))$,
and $x$ is the unique vertex in $N_{\clc(G)}(a_3)\cap V(R)$. Thus
$N_R(a_2)\cup N_R(a_3)=\{x\}$.
\end{proof}

We denote by $I_R$ the set of interior vertices of $R$, and by $F_R$
the set of frontier vertices of $R$.

\begin{claim}\cite[Claim 2 in the proof of Lemma 3]{NingLiZhang}\\
Let $x,y$ be two vertices in $R$. \\
(1) If $\{x,y\}$ is a heavy pair of $G$, then $x,y$ have two common
neighbors in $I_R$. \\
(2) If $x,y\in I_R\cup\{a\}$, $xy\in E(G)$ and $d(x)+d(y)\geq n$,
then $x,y$ have a common neighbor in $I_R$.
\end{claim}

\begin{proof}
(1) Note that every vertex in $F_R$ has at least one neighbor in
$G-R$, and every vertex in $G-R$ has at most one neighbor in $F_R$.
We have $|N_{G-R}(F_R\backslash\{x,y\})|\geq|F_R\backslash\{x,y\}|$.
Also note that
$n=|I_R\backslash\{x,y\}|+|F_R\backslash\{x,y\}|+|V(G-R)|+2$. Thus
\begin{align*}
n   & \leq d(x)+d(y)\\
    & =d_{I_R}(x)+d_{I_R}(y)+d_{F_R}(x)+d_{F_R}(y)+d_{G-R}(x)+d_{G-R}(y)\\
    & \leq d_{I_R}(x)+d_{I_R}(y)+2|F_R\backslash\{x,y\}|+
        d_{G-R}(x)+d_{G-R}(y)\\
    & \leq d_{I_R}(x)+d_{I_R}(y)+|F_R\backslash\{x,y\}|+
        |N_{G-R}(F_R\backslash\{x,y\})|+|N_{G-R}(x)|+|N_{G-R}(y)|\\
    & =d_{I_R}(x)+d_{I_R}(y)+|F_R\backslash\{x,y\}|+|N_{G-R}(F_R)|\\
    & \leq d_{I_R}(x)+d_{I_R}(y)+|F_R\backslash\{x,y\}|+|V(G-R)|,
\end{align*}
and $$d_{I_R}(x)+d_{I_R}(y)\geq
n-|F_R\backslash\{x,y\}|-|V(G-R)|=|I_R\backslash\{x,y\}|+2.$$ This
implies that $x,y$ have two common neighbors in $I_R$.

(2) Note that if $a_2,a_3\in N_{G-R}(R)$, then they have a common
neighbor in $F_R\backslash\{a\}$. By Claim 1, we can see that
$$|V(G-R)|\geq|F_R|+1 \mbox{ and } |V(G-R)\backslash
N_{G-R}(a)|\geq|F_R\backslash\{a\}|+1.$$

If $x,y\in I_R$, then
\begin{align*}
n   & \leq d(x)+d(y)\\
    & =d_{I_R}(x)+d_{I_R}(y)+d_{F_R}(x)+d_{F_R}(y)\\
    & \leq d_{I_R}(x)+d_{I_R}(y)+2|F_R|\\
    & \leq d_{I_R}(x)+d_{I_R}(y)+|F_R|+|V(G-R)|-1,
\end{align*}
and $$d_{I_R}(x)+d_{I_R}(y)\geq n-|F_R|-|V(G-R)|+1=|I_R|+1.$$ This
implies that $x,y$ have a common neighbor in $I_R$.

If one of $x,y$, say $y$, is equal to $a$, then
\begin{align*}
n   & \leq d(x)+d(a)\\
    & =d_{I_R}(x)+d_{I_R}(a)+d_{F_R}(x)+d_{F_R}(a)+d_{G-R}(a)\\
    & \leq d_{I_R}(x)+d_{I_R}(a)+|F_R|+|F_R\backslash\{a\}|+d_{G-R}(a)\\
    & \leq d_{I_R}(x)+d_{I_R}(a)+|F_R|+|V(G-R)\backslash N_{G-R}(a)|-1+
        |N_{G-R}(a)|\\
    & \leq d_{I_R}(x)+d_{I_R}(a)+|F_R|+|V(G-R)|-1,
\end{align*}
and $$d_{I_R}(x)+d_{I_R}(a)\geq n-|F_R|-|V(G-R)|+1=|I_R|+1.$$ This
implies that $x,a$ have a common neighbor in $I_R$.
\end{proof}
From here, the main difference between the proof presented here
and the proof of Lemma 3 in \cite{NingLiZhang} would occur when we
find an induced $Z_i$ and use the $Z_i$-c-heavy
condition.

By Lemma \ref{LeRegion}, $G$ has an induced path $P$ from $a$ to
$a_i$ such that every vertex of $P$ is either in $\{a_j: 0\leq j\leq
i\}$ or an interior vertex of some regions (we set $a_0=a$). Let
$a,a'_1,a'_2,\ldots,a'_i$ be the first $i+1$ vertices of $P$. Note
that every vertex $a'_i$ is nonadjacent to every vertex in
$\{b,c\}\cup I_R$. If $abca$ is also a triangle in $G$, then
$\{a,b,c,a'_1,\ldots,a'_i\}$ induces a $Z_i$ in $G$. Thus one vertex
of $\{b,c\}$ is heavy in $G$ and one of $\{a'_{i-1},a'_i\}$ is heavy
in $G$. We assume without loss of generality that $b,a'_{i-1}$ are
heavy in $G$, and then, also are heavy in $\clc(G)$. Then
$\{b,a'_{i-1}\}$ is a heavy pair in $\clc(G)$, a contradiction. So
we only consider the case one edge of $\{ab,bc,ac\}$ does not exist
in $G$.

If $I_R=\emptyset$, then $R$ is a clique in $G$, and $ab,bc,ac\in
E(G)$, a contradiction. Thus, $I_R\neq\emptyset$. By Lemma
\ref{LeRegion}, $a$ has a neighbor in $I_R$.

\begin{claim}\cite[Claim 3 in the proof of Lemma 3]{NingLiZhang}\\
$d_{I_R}(a)=1$.
\end{claim}

\begin{proof}
If $a$ is contained in a triangle $axya$ such that $x,y\in I_R$,
then $\{a,x,y,a'_1,\ldots,\break a'_i\}$ induces a $Z_i$ in $G$. Thus one
vertex of $\{x,y\}$ is heavy in $G$ and one vertex of
$\{a'_{i-1},a'_i\}$ is heavy in $G$, a contradiction. Hence,
$N_{I_R}(a)$ is an independent set.

Suppose that $d_{I_R}(a)\geq 2$. Let $x,y$ be two vertices in
$N_{I_R}(a)$. Then $xy\notin E(G)$. Since $\{a,x,y,a'_1\}$ induces a
claw in $G$, and $\{a'_1,x\}$, $\{a'_1,y\}$ are not heavy pairs of
$G$, it follows $\{x,y\}$ is a heavy pair of $G$. Without loss of
generality, suppose that $x$ is heavy in $G$.

If $a$ is also heavy in $G$, then by Claim 2, $a,x$ have a common
neighbor in $I_R$, contradicting the fact that $N_{I_R}(a)$ is
independent. So we conclude that $a$ is light in $G$.

Since $\{x,y\}$ is a heavy pair of $G$, by Claim 2, $x,y$ have two
common neighbors in $I_R$. Let $x',y'$ be two vertices in
$N_{I_R}(x)\cap N_{I_R}(y)$. Clearly $ax',ay'\notin E(G)$.

If $x'y'\in E(G)$, then $\{x,x',y',a,a'_1,\ldots,a'_{i-1}\}$ induces
a $Z_i$ in $G$. Thus one vertex of $\{a'_{i-2},a'_{i-1}\}$ is heavy
in $G$. This implies either $\{x,a'_{i-2}\}$ or $\{x,a'_{i-1}\}$ is
a heavy pair of $G$, and also a heavy pair of $\clc(G)$, a
contradiction. So we conclude that $x'y'\notin E(G)$.

Note that $\{x,x',y',a\}$ induces a claw in $G$, and $a$ is light in
$G$. So one vertex of $\{x',y'\}$ is heavy in $G$. We assume without
loss of generality that $x'$ is heavy in $G$. By Claim 2, $x,x'$
have a common neighbor $x''$ in $I_R$. Clearly $ax''\notin E(G)$.
Thus $\{x,x',x'',a,a'_1,\ldots,a'_{i-1}\}$ induces a $Z_i$, and
hence one vertex of $\{a'_{i-2},a'_{i-1}\}$ is heavy in $G$, a
contradiction.
\end{proof}

Now let $x$ be the vertex in $N_{I_R}(a)$. The left part is
almost the same as in the proof of Lemma 3 in \cite{NingLiZhang}.
We rewrite it here.

\begin{claim}\cite[Claim 4 in the proof of Lemma 3]{NingLiZhang}\\
$N_R(a)=V(R)\backslash\{a\}$.
\end{claim}

\begin{proof}
Suppose that $V(R)\backslash\{a\}\backslash N_R(a)\neq\emptyset$. By
Lemma \ref{LeRegion}, $R-x$ is connected. Let $y$ be a vertex in
$V(R)\backslash\{a\}\backslash N_R(a)$ such that $a,y$ have a common
neighbor $z$ in $R-x$. Since $N_{I_R}(a)=\{x\}$ and $z\in
N_R(a)\backslash\{x\}$, $z$ is a frontier vertex of $R$. Let $z'$ be
a vertex in $N_{G-R}(z)$. Then $\{z,y,a,z'\}$ induces a claw in $G$.
Since $\{a,z'\}$, $\{y,z'\}$ are not heavy pairs of $G$, $\{a,y\}$
is a heavy pair of $G$. By Claim 2, $a,y$ have two common neighbors
in $I_R$, contradicting Claim 3.
\end{proof}

By Claims 3 and 4, we can see that $|I_R|=1$. Recall that one edge
of $\{ab,bc,ac\}$ is not in $E(G)$. By Claim 4, $ab,ac\in E(G)$.
This implies that $bc\notin E(G)$, and $\{a,b,c,a'_1\}$ induces a
claw in $G$. Since $\{b,a'_1\}$, $\{c,a'_1\}$ are not heavy pairs of
$G$, $\{b,c\}$ is a heavy pair of $G$. By Claim 2, $b,c$ have two
common neighbors in $I_R$, contradicting the fact that $|I_R|=1$.
\end{proof}

\begin{corollary}\label{CoSZi}
For $i=1$ or $i\geq 3$, the class of claw-o-heavy and $Z_i$-c-heavy
graphs is c-stable.
\end{corollary}

\begin{theorem}\label{ThSA}
Let $S$ be a connected claw-free and $K_{1,1,2}$-free graph of order
at least 3. Then the class of claw-o-heavy and $S$-c-heavy graphs is
c-stable, if and only if
$$S\in\{K_i: i\geq 3\}\cup\{P_i: i\geq 3\}
\cup\{Z_i: i=1 \mbox{ or } i\geq 3\}.$$
\end{theorem}

\begin{proof}
If $S=K_i$, $i\geq 3$, then every graph is $S$-c-heavy, and the
class of claw-o-heavy and $S$-c-heavy graphs is c-stable. If
$S=P_i$, $i\geq 3$ or $S=Z_i$, $i=1$ or $i\geq 3$, then by
Corollaries \ref{CoSPi} and \ref{CoSZi}, the class of claw-o-heavy
and $S$-c-heavy graphs is c-stable. This completes the `if' part of
the proof.

Now we consider the `only if' part of the theorem. We first
construct some claw-o-heavy graphs as in Figure 6.

\begin{center}
\setlength{\unitlength}{0.90pt} \small
\begin{picture}(460,240)

\put(0,120){\thinlines \put(60,70){\circle{80}} \put(55,65){$K_r$}
\multiput(140,30)(0,60){2}{\multiput(0,0)(0,20){2}{\circle*{4}}}
\put(140,30){\line(-1,0){80}} \qbezier(140,30)(100,110)(60,110)
\qbezier(140,50)(100,30)(60,30) \qbezier(140,50)(100,110)(60,110)
\qbezier(140,90)(100,30)(60,30) \qbezier(140,90)(100,110)(60,110)
\qbezier(140,110)(100,30)(60,30) \put(140,110){\line(-1,0){80}}
\thicklines \qbezier[4](140,50)(140,70)(140,90)

\put(180,30){\circle*{4}} \put(180,110){\circle*{4}}
\put(220,30){\circle*{4}} \put(220,110){\circle*{4}}
\put(180,30){\line(-1,0){40}} \put(180,30){\line(-2,1){40}}
\put(180,30){\line(-2,3){40}} \put(180,30){\line(-1,2){40}}
\put(180,110){\line(-1,0){40}} \put(180,110){\line(-2,-1){40}}
\put(180,110){\line(-2,-3){40}} \put(180,110){\line(-1,-2){40}}
\put(180,30){\line(0,1){80}} \put(180,30){\line(1,0){40}}
\put(180,110){\line(1,0){40}} \put(220,30){\line(0,1){80}}
\put(144,32){$a_1$} \put(144,49){$a_2$} \put(144,87){$a_{r-1}$}
\put(144,104){$a_r$} \put(182,32){$b_1$} \put(182,104){$b_2$}
\put(210,32){$c_1$} \put(210,104){$c_2$}

\put(100,10){$G_1$ ($r\geq 3$)}}

\put(240,120){\thinlines \put(60,70){\circle{80}} \put(55,65){$K_r$}
\multiput(140,30)(0,50){2}{\multiput(0,0)(0,30){2}{\circle*{4}}}
\put(140,30){\line(-1,0){80}} \qbezier(140,30)(100,110)(60,110)
\qbezier(140,60)(100,30)(60,30) \qbezier(140,60)(100,110)(60,110)
\qbezier(140,80)(100,30)(60,30) \qbezier(140,80)(100,110)(60,110)
\qbezier(140,110)(100,30)(60,30) \put(140,110){\line(-1,0){80}}
\put(185,45){\circle{30}} \put(185,95){\circle{30}}
\put(140,30){\line(1,0){45}} \qbezier(140,30)(170,60)(185,60)
\put(140,110){\line(1,0){45}} \qbezier(140,110)(170,80)(185,80)
\put(180,40){$K_k$} \put(180,90){$K_k$}

\thicklines

\multiput(160,60)(0,10){3}{\circle*{4}} \put(185,60){\circle*{4}}
\put(185,80){\circle*{4}} \put(140,30){\line(0,1){80}}
\put(160,60){\line(0,1){20}} \put(140,60){\line(1,0){20}}
\put(140,80){\line(1,0){20}} \put(185,60){\line(0,1){20}}
\put(142,32){$a_1$} \put(142,53){$a_2$} \put(142,82){$a_3$}
\put(142,103){$a_4$} \put(187,61){$b_1$} \put(163,56){$b_2$}
\put(163,79){$b_3$} \put(187,74){$b_4$}

\put(60,10){$G_2$ ($k+3\leq r\leq 2k-2$)}}

\put(0,0){\thinlines \put(60,70){\circle{80}} \put(55,65){$K_r$}
\put(140,30){\line(-1,0){80}} \qbezier(140,30)(100,110)(60,110)
\qbezier(140,60)(100,30)(60,30) \qbezier(140,60)(100,110)(60,110)
\qbezier(140,80)(100,30)(60,30) \qbezier(140,80)(100,110)(60,110)
\qbezier(140,110)(100,30)(60,30) \put(140,110){\line(-1,0){80}}

\thicklines

\multiput(140,30)(0,50){2}{\multiput(0,0)(0,30){2}{\multiput(0,0)(27,0){4}{\circle*{4}}
\put(0,0){\line(1,0){27}} \qbezier[4](27,0)(40.5,0)(54,0)
\put(54,0){\line(1,0){27}}} \put(80,0){\line(0,1){30}}
\put(80,15){\circle*{4}}} \put(140,33){$a_0$} \put(163,33){$a_1$}
\put(186,33){$a_{t-1}$} \put(211,33){$a_t$} \put(140,63){$b_0$}
\put(163,63){$b_1$} \put(186,63){$b_{t-1}$} \put(211,63){$b_t$}
\put(140,83){$c_0$} \put(163,83){$c_1$} \put(186,83){$c_{t-1}$}
\put(211,83){$c_t$} \put(140,113){$d_0$} \put(163,113){$d_1$}
\put(186,113){$d_{t-1}$} \put(211,113){$d_t$}

\put(50,10){$G_3$ ($t\geq\max\{i,j,k\}$ and $r\geq 4t$)}}

\put(240,0){\thinlines \put(60,70){\circle{80}} \put(55,65){$K_r$}
\multiput(140,30)(0,50){2}{\multiput(0,0)(0,30){2}{\circle*{4}}}
\put(140,30){\line(-1,0){80}} \qbezier(140,30)(100,110)(60,110)
\qbezier(140,60)(100,30)(60,30) \qbezier(140,60)(100,110)(60,110)
\qbezier(140,80)(100,30)(60,30) \qbezier(140,80)(100,110)(60,110)
\qbezier(140,110)(100,30)(60,30) \put(140,110){\line(-1,0){80}}

\thicklines

\put(170,30){\circle*{4}} \put(170,110){\circle*{4}}
\multiput(200,30)(0,50){2}{\multiput(0,0)(0,30){2}{\circle*{4}}}
\put(140,30){\line(1,0){60}} \put(140,110){\line(1,0){60}}
\put(170,30){\line(-1,1){30}} \put(170,30){\line(1,1){30}}
\put(170,110){\line(-1,-1){30}} \put(170,110){\line(1,-1){30}}
\put(200,30){\line(0,1){80}} \put(142,32){$a_1$} \put(143,58){$a_2$}
\put(143,77){$a_3$} \put(142,103){$a_4$} \put(165,35){$b_1$}
\put(165,100){$b_4$} \put(190,32){$c_1$} \put(189,58){$c_2$}
\put(189,77){$c_3$} \put(190,103){$c_4$}

\put(90,10){$G_4$ ($r\geq 8$)}}

\end{picture}
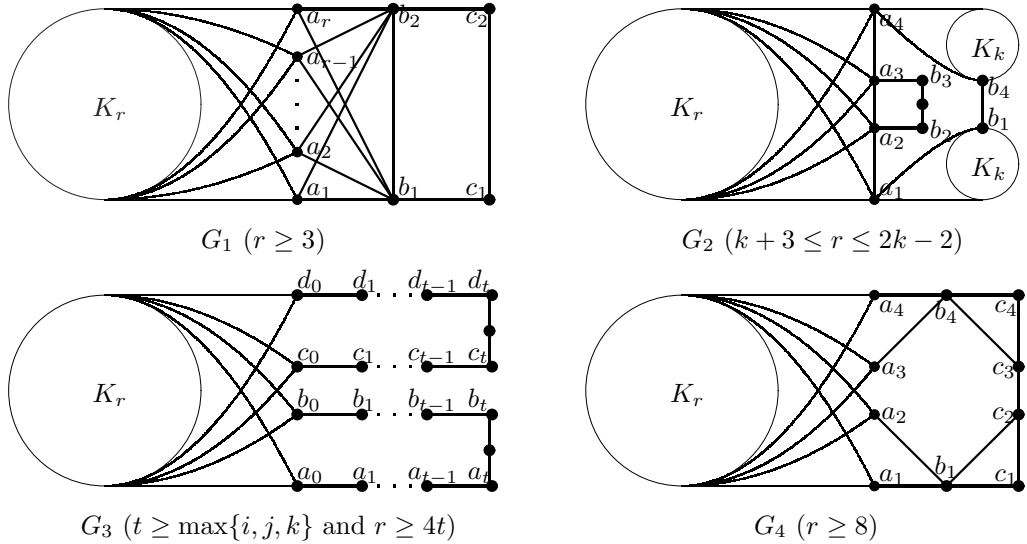
{\small Figure 6. Some claw-o-heavy graphs.}

\end{center}

Suppose $S$ is a claw-free and $K_{1,1,2}$-free graph such that the
class of claw-o-heavy and $S$-c-heavy graphs is c-stable. Consider
the case where the class of claw-free and $S$-free graphs is
r-stable. By Theorem \ref{ThBrRyFa}, $S\in\{C_3,H\}\cup\{P_i: i\geq
1\}\cup\{Z_i: i\geq 1\}\cup\{N_{i,j,k}: i,j,k\geq 1\}$. Now we will
explain why the graphs in Figure 6. are required graphs.
\begin{itemize}
\item The graph
$G_1$ is $Z_2$-c-heavy, and the closure $\clc(G_1)$ is obtained by adding
all possible edges between vertices in the
$V(K_r)\cup\{a_1,\ldots,a_r,b_1,b_2\}$. Notice that the subgraph of
$\clc(G_1)$ induced by $\{a_1,a_2,b_1,c_1,c_2\}$ is a $Z_2$ which is
not c-heavy in $\clc(G_1)$.

\item The graph $G_2$ is $N$-c-heavy, and the
closure $\clc(G_2)$ is obtained by adding all possible edges between
vertices in the $V(K_r)\cup\{a_1,\ldots,a_4\}$. Notice that the subgraph of
$\clc(G_2)$ induced by $\{a_1,b_1,a_2,b_2,a_3,b_3\}$ is an $N$ which
is not c-heavy in $\clc(G_2)$ (noting that $a_2,a_3$ are not heavy
in $\clc(G)$).

\item The graph $G_3$ is $N_{i,j,k}$-c-heavy for
$\max\{i,j,k\}\geq 2$ (in fact, it is $N_{i,j,k}$-free),
and the closure $\clc(G_3)$ is obtained by adding all possible edges between
vertices in the $V(K_r)\cup\{a_0,b_0,c_0,d_0\}$. Notice that the subgraph of
$\clc(G_3)$ induced by
$\{a_0,\ldots,a_i,\break b_0,\ldots,b_j, c_0,\ldots,c_k\}$ is an $N_{i,j,k}$
which is not c-heavy in $\clc(G_3)$.

\item The graph $G_4$ is $H$-c-heavy
($\max\{i,j,k\}\geq 2$) (in fact, it is $H$-free), and the closure
$\clc(G_4)$ is obtained by adding all possible edges between
vertices in the $V(K_r)\cup\{a_1,\ldots,a_4\}$. Notice that the subgraph of
$\clc(G_4)$ induced by $\{a_1,a_2,b_1,c_1,c_2\}$ is an $H$ which is
not c-heavy in $\clc(G_4)$.
\end{itemize}
Thus, we can see $S$ is $C_3$, $P_i$, $i\geq 1$ or
$Z_i$, $i=1$ or $i\geq 3$.

Next we consider the case where the class of claw-free and $S$-free
graphs is not r-stable. Let $G'$ be a claw-free and $S$-free graph
such that $\clr(G)$ is not $S$-free. Let $G$ be the disjoint union
of $G'$ and an empty graph of order $|V(G')|$. Clearly $G$ is
claw-free and $S$-free, and then, claw-o-heavy and $S$-c-heavy. Let
$G_i$, $1\leq i\leq r$, be the sequence of graphs in the definition
of the c-closure of $G$, where $G=G_1$ and $\clc(G)=G_r$. Note that
for every $i$, every vertex of $G_i$ has degree less than
$|V(G)|/2$. This implies that the c-eligible vertices of $G_i$ are
exactly the r-eligible ones. Thus $\clc(G)=\clr(G)$ and $\clc(G)$
contains an induced $S$. Note that $\clc(G)$ has no heavy vertex. If
$S$ has a maximal clique $C$ such that $S-C$ has a nontrivial
component, then the induced $S$ in $\clc(G)$ is not c-heavy, a
contradiction. So we conclude that for every maximal clique $C$ of
$S$, $S-C$ has only isolated vertices.

Let $C$ be a maximal clique of $S$. If $V(S)\backslash
V(C)=\emptyset$, then $S$ is a complete graph $K_k$. Now we consider
the case that $V(S)\backslash V(C)\neq\emptyset$. Note that every vertex of
$S-C$ is an isolated vertex. Let $x$ be a vertex in $S-C$. Since $C$
is a maximal clique, $C\backslash N_S(x)\neq\emptyset$. If
$|C\backslash N_S(x)|\geq 2$, then let $C'$ be a maximal clique of
$S$ containing $x$. Then $S-C'$ will have a nontrivial component, a
contradiction. So we conclude that $|C\backslash N_S(x)|=1$. Let $y$
be the vertex in $C\backslash N_S(x)$. By our assumption that $S$ is
connected, $|C|\geq 2$. If $|C|\geq 3$, letting $z,z'$ be two
vertices of $C\backslash\{y\}$, then $\{x,y,z,z'\}$ induces a
$K_{1,1,2}$ of $S$, a contradiction. Thus we conclude that $C$ has
exactly two vertices. Let $z$ be the vertex of $C$ other than $y$.
Note that $C'=C\cup\{x\}\backslash\{y\}$ is a maximal clique of $S$.
Every vertex of $S-C'$ is nonadjacent to $y$. If $S-C$ has a vertex
$w$ other than $x$, then $\{z,x,y,w\}$ induces a claw in $S$, a
contradiction. This implies that $S-C$ has only one vertex $x$, and
$S=P_3$, a contradiction.
\end{proof}

By Theorem \ref{ThSA}, the class of claw-o-heavy and $N$-c-heavy
graphs is not c-stable. However, we have a slightly larger class of graphs
which is c-stable.

Let $G$ be a graph and $M$ be an induced $N$ in $G$.
We denote the vertices of $M$ as in Figure 5. Note that $M$ is
c-heavy in $G$ if and only if there are two vertices $u,v$ of $M$ which are heavy
in $G$ such that
$\{u,v\}\notin\{\{a,a_1\},\{b,b_1\},\{c,c_1\}\}$.
Now we say that $M$ is p-heavy in $G$
if there are two vertices $u,v$ of $M$ with
$d(u)+d(v)\geq n$, such that $\{u,v\}\notin\{\{a,a_1\},\{b,b_1\},\{c,c_1\}\}$. Also, we say that $G$ is $N$-p-heavy if every
induced $N$ in $G$ is p-heavy. Note that an $N$-c-heavy graph is
also $N$-p-heavy.

Now we prove that the class of claw-o-heavy and $N$-p-heavy graphs
is c-stable.

\begin{theorem}\label{ThSN}
Let $G$ be a claw-o-heavy and $N$-p-heavy graph, and $x$ be a
c-eligible vertex of $G$. Then $G'_x$ is $N$-p-heavy.
\end{theorem}

\begin{proof}
Let $M$ be an induced $N$ in $G'_x$. We will prove that $M$ is
p-heavy. We denote the vertices of $M$ as in Figure 5. Let
$n=|V(G)|$. If $M$ is also an induced subgraph of $G$, then $M$ is
p-heavy in $G$, and then, is p-heavy in $G'_x$.

Now we consider the case $E(M)\cap B_G(x)\neq\emptyset$.
First suppose that $aa_1\in B_G(x)$. Note that $N(x)$ is a clique
in $G'_x$. This implies that $N(x)\cap V(M)=\{a,a_1\}$. Thus
$\{a,x,b,b_1,c,c_1\}$ induces an $N$ in $G$. Since $G$ is
$N$-p-heavy and $d_{G'_x}(a)\geq d_{G'_x}(x)\geq d(x)$, $M$ is
p-heavy in $G'_x$. Now we consider the case $aa_1\notin B_G(G)$, and
similarly, $bb_1,cc_1\notin B_G(G)$. Thus at least one edge in
$\{ab,ac,bc\}$ is in $B_G(x)$.

If $|B_G(x)\cap\{ab,ac,bc\}|=1$, then without loss of
generality, suppose that $ab\in B_G(x)$. Then $\{c,a,b,c_1\}$ induces a claw.
Thus one of the three pairs $\{a,b\},\{a,c_1\},\{b,c_1\}$ is a heavy
pair in $G$, and then has degree sum at least $n$ in $G'_x$. Hence
$M$ is p-heavy in $G'_x$.

If $|B_G(x)\cap\{ab,ac,bc\}|=2$, then without loss of
generality, suppose that $ab,ac\in B_G(x)$. Then $\{x,a,b,b_1,c,c_1\}$
induces an $N$. Thus there are two vertices $u,v$ in
$\{x,a,b,b_1,c,c_1\}$ such that
$\{u,v\}\notin\{\{x,a\},\{b,b_1\},\{c,c_1\}\}$, with degree sum at
least $n$ in $G$. Since $d_{G'_x}(a)\geq d(x)$, we can see that $M$
is p-heavy.

If $|B_G(x)\cap\{ab,ac,bc\}|=3$, then all the three
edges $\{ab,ac,bc\}$ are in $B_G(x)$, which implies that
$\{x,a,b,c\}$ induces a claw in $G$. So, one pair of
$\{\{a,b\},\{a,c\},\break\{b,c\}\}$ is a heavy pair in $G$, and then has
degree sum at least $n$ in $G'_x$. Hence, $M$ is p-heavy in $G'_x$.
\end{proof}

\begin{corollary}\label{CoSN}
The class of claw-o-heavy and $N$-p-heavy graphs is c-stable.
\end{corollary}

\section{Proof of Theorem \ref{ThcHeavy}}

Note that every graph is $P_3$-c-heavy and $C_3$-c-heavy, and there
indeed exist some 2-connected claw-o-heavy graphs which are not
hamiltonian. The `only if' part of the theorem can be deduced by
Theorem \ref{ThFaGo} immediately. Now we prove the `if' part of the
theorem.

\subsection*{The cases $S=P_4$, $P_5$, $P_6$.}

Note that every $P_4$-c-heavy graph is $P_5$-c-heavy and every
$P_5$-c-heavy graph is $P_6$-c-heavy. We only need to prove the case
$S=P_6$.

Let $G$ be a claw-o-heavy and $P_6$-c-heavy graph. By Theorem
\ref{ThCa} and Corollary \ref{CoSPi}, $\clc(G)$ is claw-free and
$P_6$-free. By Theorem \ref{ThDuBrBeFa}, $\clc(G)$ is hamiltonian,
and by Theorem \ref{ThCa}, so is $G$.

\subsection*{The cases $S=Z_1$, $B$, $N$.}

Note that every $Z_1$-c-heavy graph is $B$-c-heavy and every
$B$-c-heavy graph is $N$-c-heavy. We only need deal with the case
$S=N$.

Let $G$ be a claw-o-heavy and $N$-c-heavy graph. Note that every
$N$-c-heavy graph is also $N$-p-heavy. By Theorem \ref{ThCa} and
Corollary \ref{CoSN}, $\clc(G)$ is claw-free and $N$-p-heavy. If
$\clc(G)$ is hamiltonian, then so is $G$. So we assume that
$\clc(G)$ is not hamiltonian. Since $\clc(G)$ is 2-connected and
claw-free, by Theorem \ref{ThBr}, $\clc(G)$ has an induced subgraph
in $\mathcal{P}$. We denote the notation $a_i,b_i$ $i=1,2,3$ as in
Section 2 and let $n=|V(G)|$.

Note that $\clc(G)$ has no heavy pair. Since $\clc(G)$ is
$N$-p-heavy, every induced $N$ of $\clc(G)$ has two vertices in its
triangle with degree sum at least $n$. Since both triangles
$a_1a_2a_3a_1$ and $b_1b_2b_3b_1$ are contained in some induced $N$
of $\clc(G)$, two vertices of $\{a_1,a_2,a_3\}$ have degree sum at
least $n$ and two vertices of $\{b_1,b_2,b_3\}$ have degree sum at
least $n$. We assume without loss of generality that $a_1$ has the
maximum degree in $\clc(G)$ among all the six vertices. Then two
pairs of $\{\{a_1,b_1\},\{a_1,b_2\},\{a_1,b_3\}\}$ have degree sum
at least $n$. Since $a_1$ is nonadjacent to $b_2,b_3$, $\clc(G)$ has
a heavy pair, a contradiction.

\subsection*{The cases $S=Z_2$, $W$.}
Note that every $Z_2$-c-heavy graph is $W$-c-heavy. We only need to
prove the case $S=W$. If $G$ is $W$-c-heavy, then it is also $W$-o-heavy. By Theorem
\ref{ThLiRyWaZh}, $G$ is hamiltonian.

\subsection*{The case $S=Z_3$.}

Let $G$ be a claw-o-heavy and $Z_3$-c-heavy graph. By Theorem
\ref{ThCa} and Theorem \ref{ThSZi}, $\clc(G)$ is claw-free and
$Z_3$-free. By Theorem \ref{ThDuBrBeFa}, $\clc(G)$ is hamiltonian or
$\clc(G)=L_1$ or $L_2$ (see Figure 1). If $\clc(G)=L_1$ or $L_2$,
then $G$ has no c-eligible vertices (any c-eligible vertex of $G$ is
an interior vertex and of degree at least 3 in $\clc(G)$). Thus
$G=\clc(G)=L_1$ or $L_2$, contradicting the assumption $n\geq 10$.

\section{One remark}
In fact, in this paper we prove the following theorem, which is a
common extension of the case $S=N$ in Theorems \ref{ThLiRyWaZh},
\ref{ThNiZh} and \ref{ThcHeavy}.

\begin{theorem}\label{ThNpHeavy}
Let $G$ be a 2-connected graph. If $G$ is claw-o-heavy and
$N$-p-heavy, then $G$ is hamiltonian.
\end{theorem}

\section*{Acknowledgements}
Some results in this paper were reported at the conference CSGT 2014
in Teplice nad Be\v{c}vou by the first author, and he is grateful to
the organizers for a friendly atmosphere.

\end{document}